\documentclass[platex]{amsart}
\usepackage{amssymb,amsmath,amsthm,empheq,comment,url}
\usepackage[utf8]{inputenc}
\usepackage{enumerate}

\usepackage[top=30truemm,bottom=30truemm,left=20truemm,right=20truemm]{geometry}

\theoremstyle{definition}
\newtheorem{theorem}{Theorem}[section]
\newtheorem{proposition}[theorem]{Proposition}

\newtheorem{lemma}[theorem]{Lemma}

\newtheorem*{remark*}{Remark}

\DeclareMathOperator{\re}{Re}
\DeclareMathOperator{\im}{Im}

\DeclareMathOperator{\Mod}{Mod}

\title[Minimal mass blow-up solutions for NLS with a singlar potential]{Minimal mass blow-up solutions for nonlinear Schr\"{o}dinger equations with a Hartree nonlinearity}
\author[N. Matsui]{Naoki Matsui}
\date{\today}
\address[N. Mastui]{Department of Mathematics\\ Tokyo University of Science\\ 1-3 Kagurazaka, Shinjuku-ku, Tokyo 162-8601, Japan}
\email[N. Matsui]{1120703@ed.tus.ac.jp}

\keywords{nonlinear Schr\"{o}dinger equation, critical exponent, critical mass, minimal mass blow-up, blow-up rate, Hartree nonlinearity.}
\subjclass[2010]{35Q55}

\begin{document}
\maketitle

\begin{abstract}
We consider the following nonlinear Schr\"{o}dinger equation with a Hartree nonlinearity:
\[
i\frac{\partial u}{\partial t}+\Delta u+|u|^{\frac{4}{N}}u\pm\left(\frac{1}{|x|^{2\sigma}}\star|u|^2\right)u=0
\]
in $\mathbb{R}^N$. We are interested in the existence and behaviour of minimal mass blow-up solutions. Previous studies have shown the existence of minimal mass blow-up solutions with an inverse power potential and investigated the behaviour of the solution. In this paper, we investigate Hartree nonlinearity, which is a nonlinear term similar to the inverse power-type potential in terms of scaling.
\end{abstract}

\section{Introduction}
We consider the following nonlinear Schr\"{o}dinger equation with a Hartree nonlinearity:
\begin{align}
\label{NLS}
i\frac{\partial u}{\partial t}+\Delta u+|u|^{\frac{4}{N}}u\pm\left(\frac{1}{|x|^{2\sigma}}\star|u|^2\right)u=0,\quad (t,x)\in\mathbb{R}\times\mathbb{R}^N,
\end{align}
where $N\in\mathbb{N}$ and $\sigma\in\mathbb{R}$. It is well known that if
\begin{align}
\label{index1}
0<\sigma<\min\left\{\frac{N}{2},2\right\},
\end{align}
then \eqref{NLS} is locally well-posed in $H^1(\mathbb{R}^N)$ from \cite[Proposition 3.2.5, Proposition 3.2.9, Theorem 3.3.9, and Proposition 4.2.3]{CSSE}. This means that for any initial value $u_0\in H^1(\mathbb{R}^N)$, there exists a unique maximal solution $u\in C((T_*,T^*),H^1(\mathbb{R}^N))\cap C^1((T_*,T^*),H^{-1}(\mathbb{R}^N))$ for \eqref{NLS} with $u(0)=u_0$. Moreover, the mass (i.e., $L^2$-norm) and energy $E$ of the solution $u$  are conserved by the flow, where 
\[
E(u):=\frac{1}{2}\left\|\nabla u\right\|_2^2-\frac{1}{2+\frac{4}{N}}\left\|u\right\|_{2+\frac{4}{N}}^{2+\frac{4}{N}}\mp\frac{1}{4}\int_{\mathbb{R}^N}\left(\frac{1}{|x|^{2\sigma}}\star|u|^2\right)(x)|u(x)|^2dx.
\]
Furthermore, the blow-up alternative holds:
\[
T^*<\infty\quad \mbox{implies}\quad \lim_{t\nearrow T^*}\left\|\nabla u(t)\right\|_2=\infty.
\]

We define $\Sigma^k$ by
\[
\Sigma^k:=\left\{u\in H^k\left(\mathbb{R}^N\right)\ \middle|\ |x|^ku\in L^2\left(\mathbb{R}^N\right)\right\},\quad \|u\|_{\Sigma^k}^2:=\|u\|_{H^k}^2+\||x|^ku\|_2^2.
\]
Particularly, $\Sigma^1$ is called the virial space. If $u_0\in \Sigma^1$, then the solution $u$ for \eqref{NLS} with $u(0)=u_0$ belongs to $C((T_*,T^*),\Sigma^1)$ from \cite[Lemma 6.5.2]{CSSE}.

Moreover, we consider the case
\begin{align}
\label{index2}
0<\sigma<\min\left\{\frac{N}{2},1\right\}.
\end{align}
If $u_0\in H^2(\mathbb{R}^N)$, then the solution $u$ for \eqref{NLS} with $u(0)=u_0$ belongs to $C((T_*,T^*),H^2(\mathbb{R}^N))\cap C^1((T_*,T^*),L^2(\mathbb{R}^N))$ and $|x|\nabla u\in C((T_*,T^*),L^2(\mathbb{R}^N))$ from \cite[Theorem 5.3.1]{CSSE}. Furthermore, if $u_0\in \Sigma^2$, then the solution $u$ for \eqref{NLS} with $u(0)=u_0$ belongs to $C((T_*,T^*),\Sigma^2)\cap C^1((T_*,T^*),L^2(\mathbb{R}^N))$ and $|x|\nabla u\in C((T_*,T^*),L^2(\mathbb{R}^N))$ from the same proof as in \cite[Lemma 6.5.2]{CSSE}.

\subsection{Previous results}
Firstly, we describe the results regarding the mass-critical problem:
\begin{align}
\label{CNLS}
i\frac{\partial u}{\partial t}+\Delta u+|u|^{\frac{4}{N}}u=0,\quad (t,x)\in\mathbb{R}\times\mathbb{R}^N,
\end{align}
In particular, \eqref{NLS} with $\sigma=0$ is reduced to \eqref{CNLS}.

It is well known (\cite{BLGS,KGS,WGS}) that there exists a unique classical solution $Q$ for
\[
-\Delta Q+Q-\left|Q\right|^{\frac{4}{N}}Q=0,\quad Q\in H^1(\mathbb{R}^N),\quad Q>0,\quad Q\mathrm{\ is\ radial},
\]
which is called the ground state. If $\|u\|_2=\|Q\|_2$ ($\|u\|_2<\|Q\|_2$, $\|u\|_2>\|Q\|_2$), we say that $u$ has the \textit{critical mass} (\textit{subcritical mass}, \textit{supercritical mass}, respectively).

We note that $E_{\mathrm{crit}}(Q)=0$, where $E_{\mathrm{crit}}$ is the energy with respect to \eqref{CNLS}. Moreover, the ground state $Q$ attains the best constant in the Gagliardo-Nirenberg inequality
\[
\left\|v\right\|_{2+\frac{4}{N}}^{2+\frac{4}{N}}\leq\left(1+\frac{2}{N}\right)\left(\frac{\left\|v\right\|_2}{\left\|Q\right\|_2}\right)^{\frac{4}{N}}\left\|\nabla v\right\|_2^2\quad\mbox{for }v\in H^1(\mathbb{R}^N).
\]
Therefore, for all $v\in H^1(\mathbb{R}^N)$,
\[
E_{\mathrm{crit}}(v)\geq \frac{1}{2}\left\|\nabla v\right\|_2^2\left(1-\left(\frac{\left\|v\right\|_2}{\left\|Q\right\|_2}\right)^{\frac{4}{N}}\right)
\]
holds. This inequality and the mass and energy conservations imply that any subcritical mass solution for \eqref{CNLS} is global and bounded in $H^1(\mathbb{R}^N)$.

Regarding the critical mass case, we apply the pseudo-conformal transformation
\[
u(t,x)\ \mapsto\ \frac{1}{\left|t\right|^\frac{N}{2}}u\left(-\frac{1}{t},\pm\frac{x}{t}\right)e^{i\frac{\left|x\right|^2}{4t}}
\]
to the solitary wave solution $u(t,x):=Q(x)e^{it}$. Then we obtain
\[
S(t,x):=\frac{1}{\left|t\right|^\frac{N}{2}}Q\left(\frac{x}{t}\right)e^{-\frac{i}{t}}e^{i\frac{\left|x\right|^2}{4t}},
\]
which is also a solution for \eqref{CNLS} and satisfies
\[
\left\|S(t)\right\|_2=\left\|Q\right\|_2,\quad \left\|\nabla S(t)\right\|_2\sim\frac{1}{\left|t\right|}\quad (t\nearrow 0).
\]
Namely, $S$ is a minimal mass blow-up solution for \eqref{CNLS}. Moreover, $S$ is the only finite time blow-up solution for \eqref{CNLS} with critical mass, up to the symmetries of the flow (see \cite{MMMB}).

Regarding the supercritical mass case, there exists a solution $u$ for \eqref{CNLS} such that
\[
\left\|\nabla u(t)\right\|_2\sim\sqrt{\frac{\log\bigl|\log\left|T^*-t\right|\bigr|}{T^*-t}}\quad (t\nearrow T^*)
\]
(see \cite{MRUPB,MRUDB}).

Le Coz, Martel, and Rapha\"{e}l \cite{LMR} based on the methodology of \cite{RSEU} obtains the following results for
\begin{align}
\label{DPNLS}
i\frac{\partial u}{\partial t}+\Delta u+|u|^{\frac{4}{N}}u\pm|u|^{p-1}u=0,\quad (t,x)\in\mathbb{R}\times\mathbb{R}^N.
\end{align}
In \cite{N,NI}, the methods used in \cite{LMR} are improved and in \cite{NDP}, the result of \cite{LMR} are strengthened.

\begin{theorem}[\cite{LMR,NDP}]
\label{theorem:LMR}
Let $1<p<1+\frac{4}{N}$, and $\pm=+$. Then for any energy level $E_0\in\mathbb{R}$, there exist $t_0<0$ and a radially symmetric initial value $u_0\in H^1(\mathbb{R}^N)$ with
\[
\|u_0\|_2=\|Q\|_2,\quad E(u_0)=E_0
\]
such that the corresponding solution $u$ for \eqref{IPNLS} with $\pm=+$ and $u(t_0)=u_0$ blows up at $t=0$. Moreover,
\[
\left\|u(t)-\frac{1}{\lambda(t)^\frac{N}{2}}P\left(t,\frac{x}{\lambda(t)}\right)e^{-i\frac{b(t)}{4}\frac{|x|^2}{\lambda(t)^2}+i\gamma(t)}\right\|_{\Sigma^1}\rightarrow 0\quad (t\nearrow 0)
\]
holds for some blow-up profile $P$ and $C^1$ functions $\lambda:(t_0,0)\rightarrow(0,\infty)$ and $b,\gamma:(t_0,0)\rightarrow\mathbb{R}$ such that
\begin{align*}
P(t)&\rightarrow Q\quad\mbox{in}\ H^1(\mathbb{R}^N),&\lambda(t)&=C_1(p)|t|^{\frac{4}{4+N(p-1)}}\left(1+o(1)\right),\\
b(t)&=C_2(p)|t|^{\frac{4-N(p-1)}{2+N(p-1)}}\left(1+o(1)\right),& \gamma(t)^{-1}&=O\left(|t|^{\frac{4-N(p-1)}{2+N(p-1)}}\right)
\end{align*}
as $t\nearrow 0$.
\end{theorem}

\begin{theorem}[\cite{LMR}]
\label{theorem:LMR2}
Let $1<p<1+\frac{4}{N}$, and $\pm=-$. If an initial value has critical mass, then the corresponding solution for \eqref{DPNLS} with $u(0)=u_0$ is global and bounded in $H^1(\mathbb{R}^N)$.
\end{theorem}

The results show that when a small perturbation term is added to the critical problem \eqref{CNLS}, the perturbation term affects the existence of the minimal mass blow-up solution and, if it exists, its behaviour near the blow-up time.

On the basis of this result, \cite{NI} obtains the following results for
\begin{align}
\label{IPNLS}
i\frac{\partial u}{\partial t}+\Delta u+|u|^{\frac{4}{N}}u\pm\frac{1}{|x|^{2\sigma}}u=0,\quad (t,x)\in\mathbb{R}\times\mathbb{R}^N.
\end{align}

\begin{theorem}[\cite{NI}]
\label{theorem:NI1}
Assume $0<\sigma<\min\left\{1,\frac{N}{4}\right\}$. Then for any energy level $E_0\in\mathbb{R}$, there exist $t_0<0$ and a radially symmetric initial value $u_0\in H^1(\mathbb{R}^N)$ with
\[
\|u_0\|_2=\|Q\|_2,\quad E(u_0)=E_0
\]
such that the corresponding solution $u$ for \eqref{IPNLS} with $\pm=+$ and $u(t_0)=u_0$ blows up at $t=0$. Moreover,
\[
\left\|u(t)-\frac{1}{\lambda(t)^\frac{N}{2}}P\left(t,\frac{x}{\lambda(t)}\right)e^{-i\frac{b(t)}{4}\frac{|x|^2}{\lambda(t)^2}+i\gamma(t)}\right\|_{\Sigma^1}\rightarrow 0\quad (t\nearrow 0)
\]
holds for some blow-up profile $P$ and $C^1$ functions $\lambda:(t_0,0)\rightarrow(0,\infty)$ and $b,\gamma:(t_0,0)\rightarrow\mathbb{R}$ such that
\begin{align*}
P(t)&\rightarrow Q\quad\mbox{in}\ H^1(\mathbb{R}^N),&\lambda(t)&=C_1(\sigma)|t|^{\frac{1}{1+\sigma}}\left(1+o(1)\right),\\
b(t)&=C_2(\sigma)|t|^{\frac{1-\sigma}{1+\sigma}}\left(1+o(1)\right),& \gamma(t)^{-1}&=O\left(|t|^{\frac{1-\sigma}{1+\sigma}}\right)
\end{align*}
as $t\nearrow 0$.
\end{theorem}

In previous results \cite{BCD,C,CN,N}, the blow-up rate of the minimal mass blow-up solutions for \eqref{CNLS} with smooth potentials do not change. However, Theorem \ref{theorem:NI1} shows that the blow-up rate of the minimal mass blow solution changes when the potential has a singularity.

On the other hand, the following result holds in \eqref{IPNLS} with $\pm=-$.

\begin{theorem}[\cite{NI}]
\label{theorem:NI2}
Assume $N\geq 2$ and $0<\sigma<\min\left\{1,\frac{N}{2}\right\}$. If $u_0\in H^1_\mathrm{rad}(\mathbb{R}^N)$ such that $\|u_0\|_2=\|Q\|_2$, the corresponding solution $u$ for \eqref{IPNLS} with $\pm=-$ and $u(0)=u_0$ is global and bounded in $H^1(\mathbb{R}^N)$.
\end{theorem}

Moreover, \cite{NIL} obtain the following result for
\begin{align}
\label{IPLNLS}
i\frac{\partial u}{\partial t}+\Delta u+|u|^{\frac{4}{N}}u\pm\frac{1}{|x|^{2\sigma}}\log|x|u=0,\quad (t,x)\in\mathbb{R}\times\mathbb{R}^N.
\end{align}

\begin{theorem}[\cite{NIL}]
\label{theorem:NIL1}
Assume  $0<\sigma<\min\left\{1,\frac{N}{4}\right\}$. Then for any energy level $E_0\in\mathbb{R}$, there exist $t_0<0$ and a radially symmetric initial value $u_0\in H^1(\mathbb{R}^N)$ with
\[
\|u_0\|_2=\|Q\|_2,\quad E(u_0)=E_0
\]
such that the corresponding solution $u$ for \eqref{NLS} with $\pm=-$ and $u(t_0)=u_0$ blows up at $t=0$. Moreover,
\[
\left\|u(t)-\frac{1}{\lambda(t)^\frac{N}{2}}P\left(t,\frac{x}{\lambda(t)}\right)e^{-i\frac{b(t)}{4}\frac{|x|^2}{\lambda(t)^2}+i\gamma(t)}\right\|_{\Sigma^1}\rightarrow 0\quad (t\nearrow 0)
\]
holds for some blow-up profile $P$ and $C^1$ functions $\lambda:(t_0,0)\rightarrow(0,\infty)$ and $b,\gamma:(t_0,0)\rightarrow\mathbb{R}$ such that
\begin{align*}
P(t)&\rightarrow Q\quad\mbox{in}\ H^1(\mathbb{R}^N),&\lambda(t)&\approx|t|^{\frac{1}{1+\sigma}}\left|\log|t|\right|^{\frac{1}{2+2\sigma}},\\
b(t)&\approx|t|^{\frac{1-\sigma}{1+\sigma}}|\log|t||^{\frac{1}{1+\sigma}},& \gamma(t)^{-1}&=O\left(|t|^{\frac{1-\sigma}{1+\sigma}}\right)
\end{align*}
as $t\nearrow 0$.
\end{theorem}

Comparing Theorem \ref{theorem:NIL1} with Theorem \ref{theorem:NI1}, we see that the strength of the singularity of the potential corresponds to the magnitude of the blow-up rate.

\begin{theorem}[\cite{NIL}]
\label{theorem:NIL2}
Assume $N\geq 2$ and \eqref{index1}. If $u_0\in H^1_\mathrm{rad}(\mathbb{R}^N)$ such that $\|u_0\|_2=\|Q\|_2$, the corresponding solution $u$ for \eqref{NLS} with $\pm=+$ and $u(0)=u_0$ is global and bounded in $H^1(\mathbb{R}^N)$.
\end{theorem}

\subsection{Main results}
Assuming \eqref{index1} with $\pm=-$ or \eqref{index2} with $\pm=+$, we can show from Gagliardo–Nirenberg inequality that subcritical solution for \eqref{NLS} are global and bounded in $H^1(\mathbb{R}^N)$.

On the other hand, regarding critical mass in \eqref{NLS} with $\pm=+$, we obtain the following result:

\begin{theorem}[Existence of a minimal-mass blow-up solution]
\label{theorem:EMBS}
Assume \eqref{index2}. Then for any energy level $E_0\in\mathbb{R}$, there exist $t_0<0$ and a radially symmetric initial value $u_0\in H^1(\mathbb{R}^N)$ with
\[
\|u_0\|_2=\|Q\|_2,\quad E(u_0)=E_0
\]
such that the corresponding solution $u$ for \eqref{NLS} with $\pm=+$ and $u(t_0)=u_0$ blows up at $t=0$. Moreover,
\[
\left\|u(t)-\frac{1}{\lambda(t)^\frac{N}{2}}P\left(t,\frac{x}{\lambda(t)}\right)e^{-i\frac{b(t)}{4}\frac{|x|^2}{\lambda(t)^2}+i\gamma(t)}\right\|_{\Sigma^1}\rightarrow 0\quad (t\nearrow 0)
\]
holds for some blow-up profile $P$ and $C^1$ functions $\lambda:(t_0,0)\rightarrow(0,\infty)$ and $b,\gamma:(t_0,0)\rightarrow\mathbb{R}$ such that
\begin{align*}
P(t)&\rightarrow Q\quad\mbox{in}\ H^1(\mathbb{R}^N),&\lambda(t)&=C_1(\sigma)|t|^{\frac{1}{1+\sigma}}\left(1+o(1)\right),\\
b(t)&=C_2(\sigma)|t|^{\frac{1-\sigma}{1+\sigma}}\left(1+o(1)\right),& \gamma(t)^{-1}&=O\left(|t|^{\frac{1-\sigma}{1+\sigma}}\right)
\end{align*}
as $t\nearrow 0$.
\end{theorem}

On the other hand, the following result holds in \eqref{NLS} with $\pm=-$.

\begin{theorem}[Non-existence of a radial minimal-mass blow-up solution]
\label{theorem:NEMBS}
Assume \eqref{index1}. If $u_0\in H^1(\mathbb{R}^N)$ such that $\|u_0\|_2=\|Q\|_2$, the corresponding solution $u$ for \eqref{NLS} with $\pm=-$ and $u(0)=u_0$ is global and bounded in $H^1(\mathbb{R}^N)$.
\end{theorem}

See Section \ref{ProofNEMBS} for the proof.

\subsection{Comments regarding the main results}
We compare Theorem \ref{theorem:NI1} and Theorem \ref{theorem:EMBS}. In terms of blow-up rate, these result behave very similarly. On the other hand, in the construction of the blow-up profile in Proposition \ref{theorem:constprof}, we do not need to pay attention to the singularity near the origin due to the smoothing effect of convolution.

We compare Theorem \ref{theorem:LMR2}, Theorem \ref{theorem:NI2}, and Theorem \ref{theorem:NEMBS}. In Theorem \ref{theorem:NEMBS}, the radial symmetry of solution is not assumed. The reason for not assuming radial symmetry in Theorem \ref{theorem:LMR2} is that $L^p$-norms are invariant with respect to translations. On the other hand, the reason for assuming radial symmetry in Theorem \ref{theorem:NI2} is that the inverse potential is not invariant with respect to translations. In the case of Theorem \ref{theorem:NEMBS}, since the energy is invariant with respect to translations, the proof can be done without assuming radial symmetry, as in the case of Theorem \ref{theorem:LMR2}. This suggests that, with respect to blow-up, Hartree nonlinearity behaves in an intermediate way between inverse potentials and power nonlinearity.

\section{Notations}
In this section, we introduce the notation used in this paper.

Let
\[
\mathbb{N}:=\mathbb{Z}_{\geq 1},\quad\mathbb{N}_0:=\mathbb{Z}_{\geq 0}.
\]
We define
\begin{align*}
(u,v)_2&:=\re\int_{\mathbb{R}^N}u(x)\overline{v}(x)dx,&\left\|u\right\|_p&:=\left(\int_{\mathbb{R}^N}|u(x)|^pdx\right)^\frac{1}{p},\\
f(z)&:=|z|^\frac{4}{N}z,&F(z)&:=\frac{1}{2+\frac{4}{N}}|z|^{2+\frac{4}{N}}\quad \mbox{for $z\in\mathbb{C}$}.
\end{align*}
By identifying $\mathbb{C}$ with $\mathbb{R}^2$, we denote the differentials of $f$ and $F$ by $df$ and $dF$, respectively. Moreover, we define
\[
G(v):=\frac{1}{4}\int_{\mathbb{R}^N}\left(\frac{1}{|x|^{2\sigma}}\star|v|^2\right)(x)|v(x)|^2dx,\quad g(v):=\left(\frac{1}{|x|^{2\sigma}}\star|v|^2\right)v\quad\mbox{for $v\in H^1(\mathbb{R})$}.
\]
We define
\[
\Lambda:=\frac{N}{2}+x\cdot\nabla,\quad L_+:=-\Delta+1-\left(1+\frac{4}{N}\right)Q^\frac{4}{N},\quad L_-:=-\Delta+1-Q^\frac{4}{N}.
\]
Namely, $\Lambda$ is the generator of $L^2$-scaling, and $L_+$ and $L_-$ come from the linearised Schr\"{o}dinger operator to close $Q$. Then
\[
L_-Q=0,\quad L_+\Lambda Q=-2Q,\quad L_-|x|^2Q=-4\Lambda Q,\quad L_+\rho=|x|^2 Q
\]
hold, where $\rho\in\mathcal{S}(\mathbb{R}^N)$ is the unique radial solution for $L_+\rho=|x|^2 Q$. Note that there exist $C_\alpha,\kappa_\alpha>0$ such that
\[
\left|\left(\frac{\partial}{\partial x}\right)^\alpha Q(x)\right|\leq C_\alpha Q(x),\quad \left|\left(\frac{\partial}{\partial x}\right)^\alpha \rho(x)\right|\leq C_\alpha(1+|x|)^{\kappa_\alpha} Q(x).
\]
for any multi-index $\alpha$. Furthermore, there exists $\mu>0$ such that for all $u\in H_{\mathrm{rad}}^1(\mathbb{R}^N)$,
\begin{align}
\label{Lcoer}
&\left\langle L_+\re u,\re u\right\rangle+\left\langle L_-\im u,\im u\right\rangle\nonumber\\
\geq&\ \mu\left\|u\right\|_{H^1}^2-\frac{1}{\mu}\left({(\re u,Q)_2}^2+{(\re u,|x|^2 Q)_2}^2+{(\im u,\rho)_2}^2\right)
\end{align}
(e.g., see \cite{MRO,MRUPB,RSEU,WL}). We denote by $\mathcal{Y}$ the set of functions $g\in C^{\infty}(\mathbb{R}^N)$ such that
\[
\exists C_\alpha,\kappa_\alpha>0,\ \left|\left(\frac{\partial}{\partial x}\right)^\alpha g(x)\right|\leq C_\alpha(1+|x|)^{\kappa_{\alpha}}Q(x)
\]
for any multi-index $\alpha$.

Finally, we use the notation $\lesssim$ and $\gtrsim$ when the inequalities hold up to a positive constant. We also use the notation $\approx$ when $\lesssim$ and $\gtrsim$ hold. Moreover, positive constants $C$ and $\epsilon$ are sufficiently large and small, respectively.

\section{Construction of a blow-up profile}
\label{sec:constprof}
In this section, we construct a blow-up profile $P$ and introduce a decomposition of functions based on the methodology in \cite{LMR,RSEU}.

Heuristically, we state the strategy. We look for a blow-up solution in the form of
\[
u(t,x)=\frac{1}{\lambda(s)^\frac{N}{2}}v\left(s,y\right)e^{-i\frac{b(s)|y|^2}{4}+i\gamma(s)},\quad y=\frac{x}{\lambda(s)},\quad \frac{ds}{dt}=\frac{1}{\lambda(s)^2},
\]
where $v$ satisfies
\begin{align}
\label{veq}
0&=i\frac{\partial v}{\partial s}+\Delta v-v+f(v)+\lambda^\alpha \left(\frac{1}{|y|^{2\sigma}}\star|v|^2\right)v\nonumber\\
&\hspace{20pt}-i\left(\frac{1}{\lambda}\frac{\partial \lambda}{\partial s}+b\right)\Lambda v+\left(1-\frac{\partial \gamma}{\partial s}\right)v+\left(\frac{\partial b}{\partial s}+b^2\right)\frac{|y|^2}{4}v-\left(\frac{1}{\lambda}\frac{\partial \lambda}{\partial s}+b\right)b\frac{|y|^2}{2}v,
\end{align}
where $\alpha=2-2\sigma$. Since we look for a blow-up solution, it may holds that $\lambda(s)\rightarrow 0$ as $s\rightarrow\infty$. Therefore, it seems that $\lambda^\alpha \left(|y|^{-2\sigma}\star|v|^2\right)v$ is ignored. By ignoring $\lambda^\alpha \left(|y|^{-2\sigma}\star|v|^2\right)v$,
\[
v(s,y)=Q(y),\quad \frac{1}{\lambda}\frac{\partial \lambda}{\partial s}+b=1-\frac{\partial \gamma}{\partial s}=\frac{\partial b}{\partial s}+b^2=0
\]
is a solution of \eqref{veq}. Accordingly, $v$ is expected to be close to $Q$. We now consider the case where $\sigma=0$, i.e., the critical problem. Then $\lambda^2v$ corresponds to the linear term with the constant coefficient and can be removed by an appropriate transformation. In other words, $\lambda^2v$ is a negligible term for the construction of minimal-mass blow-up solutions. This suggests that $\alpha=2$ may be the threshold for ignoring the term in the context of minimal-mass blow-up. Therefore, $\lambda^\alpha \left(|y|^{-2\sigma}\star|v|^2\right)v$ may become a non-negligible term if $\alpha<2$, i.e., $\sigma>0$. Also, \eqref{veq} is difficult to solve explicitly. Consequently, we construct an approximate solution $P$ that is close to $Q$ and fully incorporates the effects of $\lambda^\alpha \left(|y|^{-2\sigma}\star|v|^2\right)v$, e.g., the singularity of the origin.

For $K\in\mathbb{N}$, we define
\[
\Sigma_K:=\left\{\ (j,k)\in{\mathbb{N}_0}^2\ \middle|\ j+k\leq K\ \right\}.
\]

\begin{proposition}
\label{theorem:constprof}
Let $K\in\mathbb{N}$ be sufficiently large. Let $\lambda(s)>0$ and $b(s)\in\mathbb{R}$ be $C^1$ functions of $s$ such that $\lambda(s)+|b(s)|\ll 1$.

(i) \textit{Existence of blow-up profile.} For any $(j,k)\in\Sigma_K$, there exist $P_{j,k}^+,P_{j,k}^-\in\mathcal{Y}$, $\beta_{j,k}\in\mathbb{R}$, and $\Psi\in H^1(\mathbb{R}^N)$ such that $P$ satisfies
\[
i\frac{\partial P}{\partial s}+\Delta P-P+f(P)+\lambda^\alpha \left(\frac{1}{|y|^{2\sigma}}\star|P|^2\right)P+\theta\frac{|y|^2}{4}P=\Psi,
\]
where $\alpha=2-2\sigma$, and  $P$ and $\theta$ are defined by
\begin{align*}
P(s,y)&:=Q(y)+\sum_{(j,k)\in\Sigma_K}\left(b(s)^{2j}\lambda(s)^{(k+1)\alpha}P_{j,k}^+(y)+ib(s)^{2j+1}\lambda(s)^{(k+1)\alpha}P_{j,k}^-(y)\right),\\
\theta(s)&:=\sum_{(j,k)\in\Sigma_K}b(s)^{2j}\lambda(s)^{(k+1)\alpha}\beta_{j,k}.
\end{align*}
Moreover, for some sufficiently small $\epsilon'>0$,
\[
\left\|e^{\epsilon'|y|}\Psi\right\|_{H^1}\lesssim\lambda^\alpha\left(\left|\frac{1}{\lambda}\frac{\partial \lambda}{\partial s}+b\right|+\left|\frac{\partial b}{\partial s}+b^2-\theta\right|\right)+(b^2+\lambda^\alpha)^{K+2}
\]
holds.

(ii) \textit{Mass and energy properties of blow-up profile.} Let define
\[
P_{\lambda,b,\gamma}(s,x):=\frac{1}{\lambda(s)^\frac{N}{2}}P\left(s,\frac{x}{\lambda(s)}\right)e^{-i\frac{b(s)}{4}\frac{|x|^2}{\lambda(s)^2}+i\gamma(s)}.
\]
Then
\begin{align*}
\left|\frac{d}{ds}\|P_{\lambda,b,\gamma}\|_2^2\right|&\lesssim\lambda^\alpha\left(\left|\frac{1}{\lambda}\frac{\partial \lambda}{\partial s}+b\right|+\left|\frac{\partial b}{\partial s}+b^2-\theta\right|\right)+(b^2+\lambda^\alpha)^{K+2},\\
\left|\frac{d}{ds}E(P_{\lambda,b,\gamma})\right|&\lesssim\frac{1}{\lambda^2}\left(\left|\frac{1}{\lambda}\frac{\partial \lambda}{\partial s}+b\right|+\left|\frac{\partial b}{\partial s}+b^2-\theta\right|+(b^2+\lambda^\alpha)^{K+2}\right)
\end{align*}
hold. Moreover,
\begin{eqnarray}
\label{Eesti}
\left|8E(P_{\lambda,b,\gamma})-\||y|Q\|_2^2\left(\frac{b^2}{\lambda^2}-\frac{2\beta}{2-\alpha}\lambda^{\alpha-2}\right)\right|\lesssim\frac{\lambda^\alpha(b^2+\lambda^\alpha)}{\lambda^2}
\end{eqnarray}
holds, where 
\[
\beta:=\beta_{0,0}=\frac{2\sigma\left(\left(\frac{1}{|y|^{2\sigma}}\star Q^2\right)Q,Q\right)_2}{\||\cdot|Q\|_2^2}.
\]
\end{proposition}

\begin{proof}
See \cite{LMR,NI} for details of proofs.

We prove (i). We set
\[
Z:=\sum_{(j,k)\in\Sigma_K}b^{2j}\lambda^{k\alpha}P_{j,k}^++i\sum_{(j,k)\in\Sigma_K}b^{2j+1}\lambda^{k\alpha}P_{j,k}^-.
\]
Then $P=Q+\lambda^\alpha Z$ holds.
\begin{align*}
\Psi&:=i\frac{\partial P}{\partial s}+\Delta P-P+f(P)+\lambda^\alpha\left(\frac{1}{|y|^{2\sigma}}\star|P|^2\right)P+\theta\frac{|y|^2}{4}P,
\end{align*}
where $P_{j,k}^+,P_{j,k}^-\in\mathcal{Y}$ and $\beta_{j,k},c^+_{j,k}\in\mathbb{R}$ are to be determined.

Firstly, we have
\begin{align*}
i\frac{\partial P}{\partial s}&=-i\sum_{(j,k)\in\Sigma_K}((k+1)\alpha+2j)b^{2j+1}\lambda^{(k+1)\alpha}P_{j,k}^+\\
&\hspace{40pt}+i\sum_{j,k\geq 0}b^{2j+1}\lambda^{(k+1)\alpha}F_{j,k}^{\frac{\partial P}{\partial s},-}+\sum_{j,k\geq 0}b^{2j}\lambda^{(k+1)\alpha}F_{j,k}^{\frac{\partial P}{\partial s},+}+\Psi^\frac{\partial P}{\partial s},
\end{align*}
where
\begin{align*}
\Psi^\frac{\partial P}{\partial s}&=\left(\frac{1}{\lambda}\frac{\partial \lambda}{\partial s}+b\right)\sum_{(j,k)\in\Sigma_K}(k+1)\alpha b^{2j}\lambda^{(k+1)\alpha}(iP_{j,k}^+-bP_{j,k}^-)\nonumber\\
&\hspace{40pt}+\left(\frac{\partial b}{\partial s}+b^2-\theta\right)\sum_{(j,k)\in\Sigma_K}b^{2j-1}\lambda^{(k+1)\alpha}(2jiP_{j,k}^+-(2j+1)bP_{j,k}^-)
\end{align*}
and for $j,k\geq0$, $F_{j,k}^{\frac{\partial P}{\partial s},\pm}$ consists of $P_{j',k'}^\pm$ and $\beta_{j',k'}$ for $(j',k')\in\Sigma_K$ such that $k'\leq k-1$ and $j'\leq j+1$ or $k'\leq k$ and $j'\leq j-1$. Only a finite number of these functions are non-zero. In particular, $F_{j,k}^{\frac{\partial P}{\partial s},\pm}$ belongs to $\mathcal{Y}$ and $F_{0,0}^{\frac{\partial P}{\partial s},\pm}=0$.

Next, we have
\begin{align*}
\Delta P-P+|P|^\frac{4}{N}P=&-\sum_{(j,k)\in\Sigma_K}b^{2j}\lambda^{(k+1)\alpha}L_+P_{j,k}^+-i\sum_{(j,k)\in\Sigma_K}b^{2j+1}\lambda^{(k+1)\alpha}L_-P_{j,k}^-\\
&\hspace{40pt}+\sum_{j,k\geq 0}b^{2j}\lambda^{(k+1)\alpha}F_{j,k}^{f,+}+i\sum_{j,k\geq 0}b^{2j+1}\lambda^{(k+1)\alpha}F_{j,k}^{f,-}+\Psi^f,
\end{align*}
where
\[
\Psi^f=f(Q+\lambda^\alpha Z)-\sum_{k=0}^{K+K'+1}\frac{1}{k!}d^kf(Q)(\lambda^\alpha Z,\cdots,\lambda^\alpha Z)
\]
and for $j,k\geq0$, $F_{j,k}^{f,\pm}$ consists of $Q$, $P_{j',k'}^\pm$, and $\beta_{j',k'}$ for $(j',k')\in\Sigma_K$ such that $k'\leq k-1$ and $j'\leq j$. Only a finite number of these functions are non-zero. In particular, $F_{j,k}^{f,\pm}$ belongs to $\mathcal{Y}$ and $F_{0,0}^{f,\pm}=0$.

Next, we have
\[
\lambda^\alpha \left(\frac{1}{|y|^{2\sigma}}\star|P|^2\right)P=\sum_{j+k\geq0}\left(b^{2j}\lambda^{(k+1)\alpha}F_{j,k}^{\sigma,+}+ib^{2j+1}\lambda^{(k+1)\alpha}F_{j,k}^{\sigma,-}\right),
\]
where for $j,k\geq 0$, $F_{j,k}^{\sigma,\pm}$ consists of $Q$, $P_{j',k'}^\pm$, and $\beta_{j',k'}$ for $(j',k')\in\Sigma_K$ such that  $k'\leq k-1$ and $j'\leq j$. Only a finite number of these functions are non-zero. In particular, $F_{j,k}^{f,\pm}$ belongs to $\mathcal{Y}$ and $F_{0,0}^{\sigma,+}=\left(\frac{1}{|x|^{2\sigma}}\star Q^2\right)Q$ and $F_{0,0}^{\sigma,-}=0$.

Finally, we have
\[
\theta\frac{|y|^2}{4}P=\sum_{(j,k)\in\Sigma_K}b^{2j}\lambda^{(k+1)\alpha}\beta_{j,k}\frac{|y|^2}{4}Q+\sum_{j,k\geq 0}b^{2j}\lambda^{(k+1)\alpha}F_{j,k}^{\theta,+}+i\sum_{j,k\geq 0}b^{2j+1}\lambda^{(k+1)\alpha}F_{j,k}^{\theta,-}
\]
and for $j,k\geq0$, $F_{j,k}^{\theta,\pm}$ consists of $Q$, $P_{j',k'}^\pm$, and $\beta_{j',k'}$ for $(j',k')\in\Sigma_K$ such that $k'\leq k-1$ and $j'\leq j$. Only a finite number of these functions are non-zero. In particular, $F_{j,k}^{\theta,\pm}$ belongs to $\mathcal{Y}$ and $F_{0,0}^{\theta,\pm}=0$.

Here, we define
\begin{align*}
F_{j,k}^{\pm}&:=F_{j,k}^{\frac{\partial P}{\partial s},\pm}+F_{j,k}^{f,\pm}+F_{j,k}^{\sigma,\pm}+F_{j,k}^{\theta,\pm},\\
\Psi^{>K}&:=\sum_{(j,k)\not\in\Sigma_K}b^{2j}\lambda^{(k+1)\alpha}F_{j,k}^++i\sum_{(j,k)\not\in\Sigma_K}b^{2j+1}\lambda^{(k+1)\alpha}F_{j,k}^-,\\
\Psi&:=\Psi^\frac{\partial P}{\partial s}+\Psi^f+\Phi^{>K}
\end{align*}
Then $\Phi^{>K}$ is a finite sum and we obtain
\begin{align*}
&i\frac{\partial P}{\partial s}+\Delta P-P+f(P)+\lambda^\alpha\left(\frac{1}{|y|^{2\sigma}}\star|P|^2\right)P+\theta\frac{|y|^2}{4}P\\
=&\sum_{(j,k)\in\Sigma_K}b^{2j}\lambda^{(k+1)\alpha}\left(-L_+P_{j,k}^++\beta_{j,k}\frac{|y|^2}{4}Q+F_{j,k}^+\right)\\
&\hspace{10pt}+i\sum_{(j,k)\in\Sigma_K}b^{2j+1}\lambda^{(k+1)\alpha}\left(-L_-P_{j,k}^--((k+1)\alpha+2j)P_{j,k}^++F_{j,k}^-\right)\\
&\hspace{20pt}+\Psi.
\end{align*}
For each $(j,k)\in\Sigma_K$, we choose recursively $P_{j,k}^\pm\in\mathcal{Y}$ and $\beta_{j,k},c_{j,k}^+\in\mathbb{R}$ that are solutions for the systems
\begin{empheq}[left={(S_{j,k})\ \empheqlbrace\ }]{align*}
&L_+P_{j,k}^+-F_{j,k}^+-\beta_{j,k}\frac{|y|^2}{4}Q=0,\\
&L_-P_{j,k}^--F_{j,k}^-+((k+1)\alpha+2j)P_{j,k}^+=0.
\end{empheq}
Such solutions $(P_{j,k}^+,P_{j,k}^-,\beta_{j,k})$ are obtained from the later Propositions \ref{theorem:Ssol}.

In the same way as \cite[Proposition 2.1]{LMR}, we have
\[
\left\|e^{\epsilon'|y|}\Psi\right\|_{H^1}\lesssim \lambda^{\alpha}\left(\left|\frac{1}{\lambda}\frac{\partial \lambda}{\partial s}+b\right|+\left|\frac{\partial b}{\partial s}+b^2-\theta\right|\right)+\left(b^2+\lambda^{\alpha}\right)^{K+2}.
\]

The rest is the same as in \cite{LMR,NI}.
\end{proof}

In the rest of this section, we construct solutions $(P_{j,k}^+,P_{j,k}^-,\beta_{j,k})\in{\mathcal{Y}}^2\times\mathbb{R}$ for systems $(S_{j,k})$ in the proof of Proposition \ref{theorem:constprof}.

\begin{proposition}
\label{theorem:Ssol}
The system $(S_{j,k})$ has a solution $(P_{j,k}^+,P_{j,k}^-,\beta_{j,k})\in\mathcal{Y}^2\times\mathbb{R}$.
\end{proposition}

\begin{proof}
We solve
\begin{empheq}[left={(S_{j,k})\ \empheqlbrace\ }]{align*}
&L_+P_{j,k}^+-F_{j,k}^+-\beta_{j,k}\frac{|y|^2}{4}Q=0,\\
&L_-P_{j,k}^--F_{j,k}^-+((k+1)\alpha+2j)P_{j,k}^+=0.
\end{empheq}

Firstly, we solve
\begin{empheq}[left={(S_{0,0})\ \empheqlbrace\ }]{align*}
&L_+P_{0,0}^+-\beta_{0,0}\frac{|y|^2}{4}Q-\left(\frac{1}{|y|^{2\sigma}}\star Q^2\right)Q=0,\\
&L_-P_{0,0}^-+\alpha P_{0,0}^+=0.
\end{empheq}
For any $\beta_{0,0}\in\mathbb{R}$, there exists a solution $P_{0,0}^+\in\mathcal{Y}$. Let
\[
\beta_{0,0}:=\frac{2\sigma\left(\left(\frac{1}{|y|^{2\sigma}}\star Q^2\right)Q,Q\right)_2}{\||\cdot|Q\|_2^2}.
\]
Then since
\begin{align*}
\left(P_{0,0}^+,Q\right)_2=-\frac{1}{2}\left\langle L_+P_{0,0}^+,\Lambda Q\right\rangle=\frac{1}{2}\left(\frac{\beta_{0,0}}{4}\||\cdot|Q\|_2^2-\frac{\sigma}{2}\left(\left(\frac{1}{|y|^{2\sigma}}\star Q^2\right)Q,Q\right)_2\right)=0,
\end{align*}
there exists a solution $P_{0,0}^-\in\mathcal{Y}$. Here, let $H(j_0,k_0)$ denote by that
\begin{align*}
&\forall (j,k)\in\Sigma_K,\ k<k_0\ \mbox{or}\ (k=k_0\ \mbox{and}\ j<j_0)\\
&\hspace{100pt}\Rightarrow (S_{j,k})\ \mbox{has a solution}\ (P_{j,k}^+,P_{j,k}^-,\beta_{j,k})\in\mathcal{Y}^2\times\mathbb{R}^2.
\end{align*}
From the above discuss, $H(1,0)$ is true. If $H(j_0,k_0)$ is true, then $F_{j_0,k_0}^\pm$ is defined and belongs to $\mathcal{Y}$. Moreover, for any $\beta_{j_0,k_0}$, there exists a solution $P_{j_0,k_0}^+$. Let be $\beta_{j_0,k_0}$ such that
\[
\left\langle-F_{j_0,k_0}^-+((k_0+1)\alpha+2j_0)P_{j_0,k_0}^+-\frac{1}{|y|^{2\sigma}}F_{j_0,k_0}^{\sigma,-},Q\right\rangle=0.
\]
Then we obtain a solution $P_{j_0,k_0}^-$.
\end{proof}

\section{Decomposition of functions}
\label{sec:decomposition}
The parameters $\tilde{\lambda},\tilde{b},\tilde{\gamma}$ to be used for modulation are obtained by the following lemma:

\begin{lemma}[Decomposition]
\label{decomposition}
There exists $\overline{l},\overline{C}>0$ such that the following statement holds. Let $I$ be an interval and $\delta>0$ be sufficiently small. We assume that $u\in C(I,H^1(\mathbb{R}^N))\cap C^1(I,H^{-1}(\mathbb{R}^N))$ satisfies
\[
\forall\ t\in I,\ \left\|\lambda(t)^{\frac{N}{2}}u\left(t,\lambda(t)y\right)e^{i\gamma(t)}-Q\right\|_{H^1}< \delta
\]
for some functions $\lambda:I\rightarrow(0,\overline{l})$ and $\gamma:I\rightarrow\mathbb{R}$. Then there exist unique functions $\tilde{\lambda}:I\rightarrow(0,\infty)$, $\tilde{b}:I\rightarrow\mathbb{R}$, and $\tilde{\gamma}:I\rightarrow\mathbb{R}\slash 2\pi\mathbb{Z}$ such that 
\begin{align}
\label{mod}
&u(t,x)=\frac{1}{\tilde{\lambda}(t)^{\frac{N}{2}}}\left(P+\tilde{\varepsilon}\right)\left(t,\frac{x}{\tilde{\lambda}(t)}\right)e^{-i\frac{\tilde{b}(t)}{4}\frac{|x|^2}{\tilde{\lambda}(t)^2}+i\tilde{\gamma}(t)},\\
&\left|\frac{\tilde{\lambda}(t)}{\lambda(t)}-1\right|+\left|\tilde{b}(t)\right|+\left|\tilde{\gamma}(t)-\gamma(t)\right|_{\mathbb{R}\slash 2\pi\mathbb{Z}}<\overline{C}\nonumber
\end{align}
hold, where $|\cdot|_{\mathbb{R}\slash 2\pi\mathbb{Z}}$ is defined by
\[
|c|_{\mathbb{R}\slash 2\pi\mathbb{Z}}:=\inf_{m\in\mathbb{Z}}|c+2\pi m|,
\]
and that $\tilde{\varepsilon}$ satisfies the orthogonal conditions
\begin{align}
\label{orthocondi}
\left(\tilde{\varepsilon},i\Lambda P\right)_2=\left(\tilde{\varepsilon},|y|^2P\right)_2=\left(\tilde{\varepsilon},i\rho\right)_2=0
\end{align}
on $I$. In particular, $\tilde{\lambda}$, $\tilde{b}$, and $\tilde{\gamma}$ are $C^1$ functions and independent of $\lambda$ and $\gamma$.
\end{lemma}

For the proof, see \cite{NI}.

\section{Approximate blow-up law}
\label{sec:apppara}
In this section, we describe the initial values and the approximation functions of the parameters $\lambda$ and $b$ in the decomposition.

We expect the parameters $\lambda$ and $b$ in the decomposition to approximately satisfy
\[
\frac{1}{\lambda}\frac{\partial \lambda}{\partial s}+b=\frac{\partial b}{\partial s}+b^2-\theta=0.
\]
Therefore, the approximation functions $\lambda_{\mathrm{app}}$ and $b_{\mathrm{app}}$ of the parameters $\lambda$ and $b$ will be determined by the following lemma:

\begin{lemma}
Let
\[
\lambda_{\mathrm{app}}(s):=\left(\frac{\alpha}{2}\sqrt{\frac{2\beta}{2-\alpha}}\right)^{-\frac{2}{\alpha}}s^{-\frac{2}{\alpha}},\quad b_{\mathrm{app}}(s):=\frac{2}{\alpha s}.
\]
Then $(\lambda_{\mathrm{app}},b_{\mathrm{app}})$ is a solution for
\[
\frac{\partial b}{\partial s}+b^2-\beta\lambda^\alpha=0,\quad \frac{1}{\lambda}\frac{\partial \lambda}{\partial s}+b=0
\]
in $s>0$.
\end{lemma}

Furthermore, the following lemma determines $\lambda(s_1)$ and $b(s_1)$ for a given energy level $E_0$ and a sufficiently large $s_1$.

\begin{lemma}
\label{paraini}
Let define $C_0:=\frac{8E_0}{\||y|Q\|_2^2}$ and $0<\lambda_0\ll 1$ such that $\frac{2\beta}{2-\alpha}+C_0{\lambda_0}^{2-\alpha}>0$. For $\lambda\in(0,\lambda_0]$, we set
\[
\mathcal{F}(\lambda):=\int_\lambda^{\lambda_0}\frac{1}{\mu^{\frac{\alpha}{2}+1}\sqrt{\frac{2\beta}{2-\alpha}+C_0\mu^{2-\alpha}}}d\mu.
\]
Then for any $s_1\gg 1$, there exist $b_1,\lambda_1>0$ such that
\[
\left|\frac{{\lambda_1}^{\frac{\alpha}{2}}}{\lambda_{\mathrm{app}}(s_1)^{\frac{\alpha}{2}}}-1\right|+\left|\frac{b_1}{b_{\mathrm{app}}(s_1)}-1\right|\lesssim {s_1}^{-\frac{1}{2}}+{s_1}^{2-\frac{4}{\alpha}},\quad \mathcal{F}(\lambda_1)=s_1,\quad E(P_{\lambda_1,b_1,\gamma})=E_0.
\]
Moreover,
\[
\left|\mathcal{F}(\lambda)-\frac{2}{\alpha\lambda^{\frac{\alpha}{2}}\sqrt{\frac{2\beta}{2-\alpha}}}\right|\lesssim\lambda^{-\frac{\alpha}{4}}+\lambda^{2-\frac{3}{2}\alpha}
\]
holds.
\end{lemma}

\begin{proof}
For the proof, see \cite{LMR,NI}.
\end{proof}

\section{Uniformity estimates for decomposition}
\label{sec:uniesti}
In this section, we estimate \textit{modulation terms}.

Let define
\[
\mathcal{C}:=\frac{\alpha}{4-\alpha}\left(\frac{\alpha}{2}\sqrt{\frac{2\beta}{2-\alpha}}\right)^{-\frac{4}{\alpha}}.
\]
For $t_1<0$ that is sufficiently close to $0$, we define
\[
s_1:=|\mathcal{C}^{-1}t_1|^{-\frac{\alpha}{4-\alpha}}.
\]
Additionally, let $\lambda_1$ and $b_1$ be given in Lemma \ref{paraini} for $s_1$ and $\gamma_1=0$. Let $u$ be the solution for \eqref{NLS} with $\pm=+$ with an initial value
\[
u(t_1,x):=P_{\lambda_1,b_1,0}(x).
\]
Then since $u$ satisfies the assumption of Lemma \ref{decomposition} in a neighbourhood of $t_1$, there exists a decomposition $(\tilde{\lambda}_{t_1},\tilde{b}_{t_1},\tilde{\gamma}_{t_1},\tilde{\varepsilon}_{t_1})$ such that $(\ref{mod})$ in a neighbourhood $I$ of $t_1$. The rescaled time $s_{t_1}$ is defined by
\[
s_{t_1}(t):=s_1-\int_t^{t_1}\frac{1}{\tilde{\lambda}_{t_1}(\tau)^2}d\tau.
\]
Then we define an inverse function ${s_{t_1}}^{-1}:s_{t_1}(I)\rightarrow I$. Moreover, we define
\begin{align*}
t_{t_1}&:={s_{t_1}}^{-1},& \lambda_{t_1}(s)&:=\tilde{\lambda}(t_{t_1}(s)),& b_{t_1}(s)&:=\tilde{b}(t_{t_1}(s)),\\
\gamma_{t_1}(s)&:=\tilde{\gamma}(t_{t_1}(s)),& \varepsilon_{t_1}(s,y)&:=\tilde{\varepsilon}(t_{t_1}(s),y).&&
\end{align*}
For the sake of clarity in notation, we often omit the subscript $t_1$. In particular, it should be noted that $u\in C((T_*,T^*),\Sigma^2(\mathbb{R}^N))$ and $|x|\nabla u\in C((T_*,T^*),L^2(\mathbb{R}^N))$. Furthermore, let $I_{t_1}$ be the maximal interval such that a decomposition as $(\ref{mod})$ is obtained and we define 
\[
J_{s_1}:=s\left(I_{t_1}\right).
\]
Additionally, let $s_0\ (\leq s_1)$ be sufficiently large and let 
\[
s':=\max\left\{s_0,\inf J_{s_1}\right\}.
\]

Let 
\[
0<M<\min\left\{\frac{1}{2},\frac{4}{\alpha}-2\right\}
\]
and $s_*$ be defined by
\[
s_*:=\inf\left\{\sigma\in(s',s_1]\ \middle|\ \mbox{(\ref{bootstrap}) holds on }[\sigma,s_1]\right\},
\]
where
\begin{align}
\label{bootstrap}
&\left\|\varepsilon(s)\right\|_{H^1}^2+b(s)^2\||y|\varepsilon(s)\|_2^2<s^{-2K},\quad \left|\frac{\lambda(s)^{\frac{\alpha}{2}}}{\lambda_{\mathrm{app}}(s)^{\frac{\alpha}{2}}}-1\right|+\left|\frac{b(s)}{b_{\mathrm{app}}(s)}-1\right|<s^{-M}.
\end{align}

Finally, we define
\[
\Mod(s):=\left(\frac{1}{\lambda}\frac{\partial \lambda}{\partial s}+b,\frac{\partial b}{\partial s}+b^2-\theta,1-\frac{\partial \gamma}{\partial s}\right).
\]

\begin{lemma}
\label{Modesti}
For $s\in(s_*,s_1]$,
\[
|(\varepsilon(s),Q)|\lesssim s^{-(K+2)},\quad |\Mod(s)|\lesssim s^{-(K+2)},\quad \|e^{\epsilon'|y|}\Psi\|_{H^1}\lesssim s^{-(K+4)}
\]
hold.
\end{lemma}

\begin{proof}
For the proof, see \cite{LMR,NI}.
\end{proof}

\section{Modified energy function}
\label{sec:MEF}
We proceed with a modified version of the technique presented in Le Coz, Martel, and Rapha\"{e}l \cite{LMR} and Rapha\"{e}l and Szeftel \cite{RSEU}. Let $m>0$ be sufficiently large and define
\begin{align*}
H(s,\varepsilon)&:=\frac{1}{2}\left\|\varepsilon\right\|_{H^1}^2+b^2\left\||y|\varepsilon\right\|_2^2-\int_{\mathbb{R}^N}\left(F(P+\varepsilon)-F(P)-dF(P)(\varepsilon)\right)dy\\
&\hspace{20pt}-\lambda^\alpha\left(G(P+\varepsilon)-G(P)-dG(P)(\varepsilon)\right),\\
S(s,\varepsilon)&:=\frac{1}{\lambda^m}H(s,\varepsilon).
\end{align*}

\begin{lemma}[Estimates of $S$]
\label{Sesti}
For $s\in(s_*,s_1]$, 
\[
\|\varepsilon\|_{H^1}^2+b^2\left\||y|\varepsilon\right\|_2^2+O(s^{-2(K+2)})\lesssim H(s,\varepsilon)\lesssim \|\varepsilon\|_{H^1}^2+b^2\left\||y|\varepsilon\right\|_2^2
\]
hold. Moreover, 
\[
\frac{1}{\lambda^m}\left(\|\varepsilon\|_{H^1}^2+b^2\left\||y|\varepsilon\right\|_2^2+O(s^{-2(K+2)})\right)\lesssim S(s,\varepsilon)\lesssim \frac{1}{\lambda^m}\left(\|\varepsilon\|_{H^1}^2+b^2\left\||y|\varepsilon\right\|_2^2\right)
\]
hold.
\end{lemma}

\begin{proof}
Since
\[
\left|\left(G(P+\varepsilon)-G(P)-dG(P)(\varepsilon)\right)\right|\lesssim \|\varepsilon\|_{H^1}^2,
\]
we obtain the conclusion as in \cite{LMR,NI}.
\end{proof}

\begin{lemma}
\label{Lambda}
For $s\in(s_*,s_1]$, 
\begin{align*}
\left|\left(f(P+\varepsilon)-f(P),\Lambda \varepsilon\right)_2\right|\lesssim \|\varepsilon\|_{H^1}^2+s^{-3K},\quad \left|\left(g(P+\varepsilon)-g(P),\Lambda \varepsilon\right)_2\right|\lesssim \|\varepsilon\|_{H^1}^2
\end{align*}
hold.
\end{lemma}

\begin{proof}
For $f$, see \cite{LMR,NI}. Similarly, we obtain.
\begin{align*}
\int_{\mathbb{R}^N}y\cdot\nabla\left(\frac{1}{4}\left(\frac{1}{|y|^{2\sigma}}\star|P+\varepsilon|^2\right)|P+\varepsilon|^2-\frac{1}{4}\left(\frac{1}{|y|^{2\sigma}}\star|P|^2\right)|P|^2-\left(\frac{1}{|y|^{2\sigma}}\star|P|^2\right)\re(P\overline{\varepsilon})\right)dy=O\left(\|\varepsilon\|_{H^1}^2\right).
\end{align*}
On the other hand, since
\[
dg(v)(w)=\left(\frac{1}{|y|^{2\sigma}}\star|v|^2\right)w+2\left(\frac{1}{|y|^{2\sigma}}\star \re(v\overline{w})\right)v,
\]
we obtain
\begin{align*}
&\int_{\mathbb{R}^N}y\cdot\nabla\left(\frac{1}{4}\left(\frac{1}{|y|^{2\sigma}}\star|P+\varepsilon|^2\right)|P+\varepsilon|^2-\frac{1}{4}\left(\frac{1}{|y|^{2\sigma}}\star|P|^2\right)|P|^2-\left(\frac{1}{|y|^{2\sigma}}\star|P|^2\right)\re(P\overline{\varepsilon})\right)dy\\
=&(g(P+\varepsilon)-g(P)-dg(P)(\varepsilon),y\cdot\nabla P)_2+\left(g(P+\varepsilon)-g(P),\Lambda\varepsilon\right)_2-\frac{N}{2}\left(g(P+\varepsilon)-g(P),\varepsilon\right)_2\\
=&\left(g(P+\varepsilon)-g(P),\Lambda\varepsilon\right)_2+O\left(\|\varepsilon\|_{H^1}^2\right).
\end{align*}
Consequently, we have the conclusion.
\end{proof}

\begin{lemma}[Derivative of $S$ in time]
\label{Sdiff}
For $s\in(s_*,s_1]$, 
\[
\frac{d}{ds}H(s,\varepsilon(s))\gtrsim -b\left(\|\varepsilon\|_{H^1}^2+b^2\left\||y|\varepsilon\right\|_2^2\right)+O(s^{-2(K+2)})
\]
holds. Moreover,
\[
\frac{d}{ds}S(s,\varepsilon(s))\gtrsim \frac{b}{\lambda^m}\left(\|\varepsilon\|_{H^1}^2+b^2\left\||y|\varepsilon\right\|_2^2+O(s^{-(2K+3)})\right)
\]
holds.
\end{lemma}

\begin{proof}
From Lemma \ref{Lambda}, we obtain the conclusion as in \cite{LMR,NI}.
\end{proof}

From Lemma \ref{Sesti} and Lemma \ref{Sdiff}, we confirm \eqref{bootstrap} on $[s_0,s_1]$. Namely, we obtain the following result:

\begin{lemma}[Re-estimation]
\label{rebootstrap}
For $s\in(s_*,s_1]$, 
\begin{align}
\label{reepsiesti}
\left\|\varepsilon(s)\right\|_{H^1}^2+b(s)^2\left\||y|\varepsilon(s)\right\|_2^2&\lesssim s^{-(2K+2)},\\
\label{reesti}
\left|\frac{\lambda(s)^{\frac{\alpha}{2}}}{\lambda_{\mathrm{app}}(s)^{\frac{\alpha}{2}}}-1\right|+\left|\frac{b(s)}{b_{\mathrm{app}}(s)}-1\right|&\lesssim s^{-\frac{1}{2}}+s^{2-\frac{4}{\alpha}}
\end{align}
holds.
\end{lemma}

\begin{proof}
For the proof, see \cite{LMR,NI}.
\end{proof}

\begin{lemma}
\label{s0s*s'}
If $s_0$ is sufficiently large, then $s_*=s'=s_0$.
\end{lemma}

\begin{proof}
This result is proven from Lemma \ref{rebootstrap} and the definitions of $s_*$ and $s'$. See \cite{N} for details of the proof.
\end{proof}

Finally, we rewrite the uniform estimates obtained for the time variable $s$ in Lemma \ref{rebootstrap} into uniform estimates for the time variable $t$.

\begin{lemma}[Interval]
\label{interval}
If $s_0$ is sufficiently large, then there is $t_0<0$ that is sufficiently close to $0$ such that for $t_1\in(t_0,0)$, 
\[
[t_0,t_1]\subset {s_{t_1}}^{-1}([s_0,s_1]),\quad \left|\mathcal{C}s_{t_1}(t)^{-\frac{4-\alpha}{\alpha}}-|t|\right|\lesssim |t|^{1+\frac{\alpha M}{4-\alpha}}\quad (t\in [t_0,t_1])
\]
holds.
\end{lemma}

\begin{proof}
For the proof, see \cite{LMR,NI}.
\end{proof}

\begin{lemma}[Conversion of estimates]
\label{uniesti}
Let
\[
\mathcal{C}_\lambda:=\mathcal{C}^{-\frac{2}{4-\alpha}}\left(\frac{\alpha}{2}\sqrt{\frac{2\beta}{2-\alpha}}\right)^{-\frac{2}{\alpha}},\quad \mathcal{C}_b:=\frac{2}{\alpha}\mathcal{C}^{-\frac{\alpha}{4-\alpha}}.
\]
For $t\in[t_0,t_1]$, 
\begin{align*}
\tilde{\lambda}_{t_1}(t)&=\mathcal{C}_\lambda|t|^\frac{2}{4-\alpha}\left(1+\epsilon_{\tilde{\lambda},t_1}(t)\right),& \tilde{b}_{t_1}(t)&=\mathcal{C}_b|t|^\frac{\alpha}{4-\alpha}\left(1+\epsilon_{\tilde{b},t_1}(t)\right),\\
\|\tilde{\varepsilon}_{t_1}(t)\|_{H^1}&\lesssim |t|^{\frac{\alpha K}{4-\alpha}},& \||y|\tilde{\varepsilon}_{t_1}(t)\|_2&\lesssim |t|^{\frac{\alpha (K-1)}{4-\alpha}}
\end{align*}
hold. Furthermore,
\[
\sup_{t_1\in[t,0)}\left|\epsilon_{\tilde{\lambda},t_1}(t)\right|\lesssim |t|^\frac{\alpha M}{4-\alpha},\quad \sup_{t_1\in[t,0)}\left|\epsilon_{\tilde{b},t_1}(t)\right|\lesssim |t|^\frac{\alpha M}{4-\alpha}.
\]
\end{lemma}

\begin{proof}
For the proof, see \cite{LMR,NI}.
\end{proof}

\section{Proof of Theorem \ref{theorem:EMBS}}
\label{sec:proof}
In this section, we complete the proof of Theorem \ref{theorem:EMBS}. See \cite{LMR,N} for details of proof.

\begin{proof}[proof of Theorem \ref{theorem:EMBS}]
Let $(t_n)_{n\in\mathbb{N}}\subset(t_0,0)$ be a monotonically increasing sequence such that $\lim_{n\nearrow \infty}t_n=0$. For each $n\in\mathbb{N}$, $u_n$ is the solution for \eqref{NLS} with $\pm=+$ with an initial value
\begin{align*}
u_n(t_n,x):=P_{\lambda_{1,n},b_{1,n},0}(x)
\end{align*}
at $t_n$, where $b_{1,n}$ and $\lambda_{1,n}$ are given by Lemma \ref{paraini} for $t_n$.

According to Lemma \ref{decomposition} with an initial value $\tilde{\gamma}_n(t_n)=0$, there exists a decomposition
\[
u_n(t,x)=\frac{1}{\tilde{\lambda}_n(t)^{\frac{N}{2}}}\left(P+\tilde{\varepsilon}_n\right)\left(t,\frac{x}{\tilde{\lambda}_n(t)}\right)e^{-i\frac{\tilde{b}_n(t)}{4}\frac{|x|^2}{\tilde{\lambda}_n(t)^2}+i\tilde{\gamma}_n(t)}.
\]
Then $(u_n(t_0))_{n\in\mathbb{N}}$ is bounded in $\Sigma^1$. Therefore, up to a subsequence, there exists $u_\infty(t_0)\in \Sigma^1$ such that
\[
u_n(t_0)\rightharpoonup u_\infty(t_0)\quad \mathrm{in}\ \Sigma^1,\quad u_n(t_0)\rightarrow u_\infty(t_0)\quad \mathrm{in}\ L^2(\mathbb{R}^N)\quad (n\rightarrow\infty),
\]
see \cite{LMR,N} for details.

Let $u_\infty$ be the solution for \eqref{NLS} with $\pm=+$ and an initial value $u_\infty(t_0)$, and let $T^*$ be the supremum of the maximal existence interval of $u_\infty$. Moreover, we define $T:=\min\{0,T^*\}$. Then for any $T'\in[t_0,T)$, $[t_0,T']\subset[t_0,t_n]$ if $n$ is sufficiently large. Then there exist $n_0$ and $C(T',t_0)>0$ such that 
\[
\sup_{n\geq n_0}\|u_n\|_{L^\infty([t_0,T'],\Sigma^1)}\leq C(T',t_0)
\]
holds. Therefore,
\[
u_n\rightarrow u_\infty\quad \mathrm{in}\ C\left([t_0,T'],L^2(\mathbb{R}^N)\right)\quad (n\rightarrow\infty)
\]
holds (see \cite{N}). In particular, $u_n(t)\rightharpoonup u_\infty(t)\ \mathrm{in}\ \Sigma^1$ for any $t\in [t_0,T)$. Furthermore, from the mass conservation, we have
\[
\|u_\infty(t)\|_2=\|u_\infty(t_0)\|_2=\lim_{n\rightarrow\infty}\|u_n(t_0)\|_2=\lim_{n\rightarrow\infty}\|u_n(t_n)\|_2=\lim_{n\rightarrow\infty}\|P(t_n)\|_2=\|Q\|_2.
\]

Based on weak convergence in $H^1(\mathbb{R}^N)$ and Lemma \ref{decomposition}, we decompose $u_\infty$ to
\[
u_\infty(t,x)=\frac{1}{\tilde{\lambda}_\infty(t)^{\frac{N}{2}}}\left(P+\tilde{\varepsilon}_\infty\right)\left(t,\frac{x}{\tilde{\lambda}_\infty(t)}\right)e^{-i\frac{\tilde{b}_\infty(t)}{4}\frac{|x|^2}{{\tilde{\lambda}_\infty(t)}^2}+i\tilde{\gamma}_\infty(t)}.
\]
Furthermore, for any $t\in[t_0,T)$, as $n\rightarrow\infty$, 
\[
\tilde{\lambda}_n(t)\rightarrow\tilde{\lambda}_\infty(t),\quad \tilde{b}_n(t)\rightarrow \tilde{b}_\infty(t),\quad e^{i\tilde{\gamma}_n(t)}\rightarrow e^{i\tilde{\gamma}_\infty(t)},\quad\tilde{\varepsilon}_n(t)\rightharpoonup \tilde{\varepsilon}_\infty(t)\quad \mathrm{in}\ \Sigma^1
\]
hold. Consequently, from the uniform estimate in Lemma \ref{uniesti}, as $n\rightarrow\infty$, we have
\begin{align*}
\tilde{\lambda}_{\infty}(t)&=\mathcal{C}_\lambda\left|t\right|^\frac{2}{4-\alpha}(1+\epsilon_{\tilde{\lambda},0}(t)),& \tilde{b}_{\infty}(t)&=\mathcal{C}_b\left|t\right|^\frac{\alpha}{4-\alpha}(1+\epsilon_{\tilde{b},0}(t)),\\
\|\tilde{\varepsilon}_{\infty}(t)\|_{H^1}&\lesssim \left|t\right|^{\frac{\alpha K}{4-\alpha}},\quad \||y|\tilde{\varepsilon}_{\infty}(t)\|_2\lesssim \left|t\right|^{\frac{\alpha (K-1)}{4-\alpha}},&\left|\epsilon_{\tilde{\lambda},0}(t)\right|&\lesssim |t|^\frac{\alpha M}{4-\alpha},\quad \left|\epsilon_{\tilde{b},0}(t)\right|\lesssim |t|^{\frac{\alpha M}{4-\alpha}}.
\end{align*}
Consequently, we obtain that $u$ converges to the blow-up profile in $\Sigma^1$.

Finally, we check energy of $u_\infty$. Since
\[
E\left(u_n\right)-E\left(P_{\tilde{\lambda}_n,\tilde{b}_n,\tilde{\gamma}_n}\right)=\int_0^1\left\langle E'(P_{\tilde{\lambda}_n,\tilde{b}_n,\tilde{\gamma}_n}+\tau \tilde{\varepsilon}_{\tilde{\lambda}_n,\tilde{b}_n,\tilde{\gamma}_n}),\tilde{\varepsilon}_{\tilde{\lambda}_n,\tilde{b}_n,\tilde{\gamma}_n}\right\rangle d\tau
\]
and $E'(w)=-\Delta w-|w|^\frac{4}{N}w-\left(|x|^{-2\sigma}\star|w|^2\right)w$, we have
\[
E\left(u_n\right)-E\left(P_{\tilde{\lambda}_n,\tilde{b}_n,\tilde{\gamma}_n}\right)=O\left(\frac{1}{{\tilde{\lambda}_n}^2}\|\tilde{\varepsilon}_n\|_{H^1}\right)=O\left(|t|^\frac{\alpha K-4}{4-\alpha}\right).
\]
Similarly, we have
\[
E\left(u_\infty\right)-E\left(P_{\tilde{\lambda}_\infty,\tilde{b}_\infty,\tilde{\gamma}_\infty}\right)=O\left(\frac{1}{{\tilde{\lambda}_\infty}^2}\|\tilde{\varepsilon}_\infty\|_{H^1}\right)=O\left(|t|^\frac{\alpha K-4}{4-\alpha}\right).
\]
From the continuity of $E$, we have
\[
\lim_{n\rightarrow \infty}E\left(P_{\tilde{\lambda}_n,\tilde{b}_n,\tilde{\gamma}_n}\right)=E\left(P_{\tilde{\lambda}_\infty,\tilde{b}_\infty,\tilde{\gamma}_\infty}\right)
\]
and from the conservation of energy,
\[
E\left(u_n\right)=E\left(u_n(t_n)\right)=E\left(P_{\tilde{\lambda}_{1,n},\tilde{b}_{1,n},0}\right)=E_0.
\]
Therefore, we have
\[
E\left(u_\infty\right)=E_0+o_{t\nearrow0}(1)
\]
and since $E\left(u_\infty\right)$ is constant for $t$, $E\left(u_\infty\right)=E_0$.
\end{proof}

\section{Proof of Theorem \ref{theorem:NEMBS}}
In this section, we describe the proof of Theorem \ref{theorem:NEMBS}.

\label{ProofNEMBS}
\begin{proof}[Proof of Theorem \ref{theorem:NEMBS}]
We assume that $u$ is a critical-mass radial solution for \eqref{NLS} with $\pm=-$ and blows up at $T^*$. Let a sequence $(t_n)_{n\in\mathbb{N}}$ be such that $t_n\rightarrow T^*$ as $n\rightarrow T^*$ and define
\[
\lambda_n:=\frac{\|\nabla Q\|_2}{\|\nabla u(t_n)\|},\quad v_n(x):={\lambda_n}^\frac{N}{2}u(t_n,\lambda_n x).
\]
Then
\[
\|v_n\|_2=\|Q\|_2,\quad \|\nabla v_n\|_2=\|\nabla Q\|_2
\]
hold. Moreover,
\[
E_0:=E(u(t_n))\geq E_{\mathrm{crit}}(u(t_n))=\frac{E_{\mathrm{crit}}(v_n)}{{\lambda_n}^2}.
\]
Therefore, we obtain
\[
\limsup_{n\rightarrow\infty}E_{\mathrm{crit}}(v_n)\leq 0.
\]
From the standard concentration argument (see \cite{MRUPB,LMR}), there exist sequences $(x_n)_{n\in\mathbb{N}}\subset\mathbb{R}^N$ and $(\gamma_n)_{n\in\mathbb{N}}\subset\mathbb{R}$ such that
\[
v_n(\cdot-x_n)e^{i\gamma_n}\rightarrow Q\quad\mbox{in}\ H^1(\mathbb{R}^N)\quad(n\rightarrow\infty).
\]

Here, since
\begin{align*}
\int_{\mathbb{R}}\left(\frac{1}{|x|^{2\sigma}}\star|v_n(\cdot-x_n)|^2\right)(x)|v_n(x-x_n)|^2dx=\int_{\mathbb{R}}\left(\frac{1}{|x|^{2\sigma}}\star|v_n|^2\right)(x)|v_n(x)|^2dx,
\end{align*}
we obtain
\[
\int_{\mathbb{R}}\left(\frac{1}{|x|^{2\sigma}}\star|u(t_n)|^2\right)(x)|u(t_n,x)|^2dx=\frac{1}{\lambda_n^{2\sigma}}\int_{\mathbb{R}}\left(\frac{1}{|x|^{2\sigma}}\star|v_n(\cdot-x_n)|^2\right)(x)|v_n(x-x_n)|^2dx.
\]
Therefore, since $E_{\mathrm{crit}}(u)\geq 0$ and $v_n(\cdot-x_n)e^{i\gamma_n}\rightarrow Q$ in $H^1(\mathbb{R}^N)$,
\[
E_0=E(u(t_n))\geq \frac{1}{4\lambda_n^{2\sigma}}\int_{\mathbb{R}}\left(\frac{1}{|x|^{2\sigma}}\star|v_n(\cdot-x_n)|^2\right)(x)|v_n(x-x_n)|^2dx\rightarrow \infty\quad(n\rightarrow\infty).
\]
It is a contradiction.
\end{proof}

\end{document}